\documentclass[reqno]{amsart}

\usepackage{amsmath, amsthm, amsfonts, eucal, amssymb}
\usepackage[english]{babel}
\usepackage{verbatim}
\usepackage[utf8]{inputenc}
\usepackage{graphicx}
\usepackage{layout}
\usepackage[all]{xy}
\usepackage{hyperref}
\usepackage[font=small,labelsep=none]{caption}
\usepackage{subcaption}
\usepackage{pinlabel}

\newtheorem{lem}{Lemma}
\newtheorem{thm}{Theorem}

\theoremstyle{definition}
\newtheorem{defn}{Definition}

\theoremstyle{remark}

\begin{document}

\title[Knot exteriors with all positive genus surfaces]{Knot exteriors with all compact surfaces of positive genus essentially embedded}

%\title[Longitudinal essential surfaces of unbounded Euler characteristics]{Longitudinal essential surfaces of independently large genus and number of boundary components}

\author{Jo\~{a}o M. Nogueira}
\address{University of Coimbra, CMUC, Department of Mathematics,  Apartado 3008, 3001-454 Coimbra, Portugal\\ nogueira@mat.uc.pt}
\thanks{This work was partially supported by the Centre for Mathematics of the University of Coimbra - UIDB/00324/2020, funded by the Portuguese Government through FCT/MCTES.  This work was also partially supported by the UT Austin | Portugal Colab. } %This work was also partially suported by the UTAustin$|$Portugal program CoLab.}

%\keywords{ Meridional essential surface; arbitrarily many boundary components; unbounded Euler characteristics; hyperbolic knot}

%\subjclass[2010]{57, 57}

\maketitle

\begin{abstract} It is well known that there exist knots with Seifert surfaces of arbitrarily high genus. In this paper, we show the existence of infinitely many knot exteriors where each of which has longitudinal essential surfaces of any positive genus and any number of boundary components.
%\subclass[2010]{57M25 \and 57N10}
\end{abstract}

\section{Introduction}
Essential surfaces have played an important role in understanding 3-manifold topology since the second half of the last century. One particularly interesting property is the existence of essential surfaces of arbitrarily large Euler characteristics in some 3-manifolds.  For knot exteriors in particular,  it is well known since the work of Lyon \cite{Lyon} that a knot exterior can have closed essential surfaces and also Seifert surfaces of arbitrarily high genus.  Many more examples of knot exteriors with these properties have been published throughout the years.  For instance,  besides the result of Lyon,  several other collections of knots have been given for which there are Seifert surfaces of arbitrarily high genus as in work of Parris \cite{Parris-78} (see also \cite{W-08}),  Gustafson \cite{Gustafson-81},  Ozawa and Tsutsumi \cite{Ozawa-Tsutsumi-03}  or Tsutsumi \cite{Tsutsumi-04}.  We will further show that a knot exterior can have a collection of longitudinal essential surfaces with arbitrarily large Euler characteristics not only due to large genus, as a collection of Seifert surfaces can have, but also from the number of boundary components.  In fact,  the collection of longitudinal essential surfaces can be with every number of boundary components and all positive genus.

\begin{thm}\label{main}
There are infinitely many knots in the 3-sphere, each of which has in its exterior a longitudinal essential surface of any positive genus and any number of boundary components. 
\end{thm}

With respect to surfaces of genus zero, we know from the work of Gabai \cite{Gabai-87} when proving the Property R and the Poenaru conjectures,  also related to the Cabling conjecture,  that there are no longitudinal essential planar surfaces in knot exteriors besides the disk bounded by the unknot.  However,  it is known by work of Hatcher and Thurston, Proposition 1(3) of \cite{H-T-85}, that there exists a collection of (non-meridional) essential surfaces with an arbitrarily large number of boundary components in some 2-bridge knot exteriors. As it is not clear in Proposition 1(3) of \cite{H-T-85} whether the surfaces are orientable, the number of boundary components for orientable surfaces can only be claimed to be even, by taking the boundary of a regular neighborhood of a non-orientable surface if that is the case.  Also, as it is well known,  composite knots do not have genus one Seifert surfaces,  and the boundary slope of a Seifert surface in the knot exterior is 0 (longitudinal).  Hence we cannot have a statement as in this theorem for composite knots or for a different boundary slope.\\
This theorem contrasts with other results in the literature.  For instance,  Wilson \cite{W-08} proved that a small knot in $S^3$,  \textit{i.e.} without closed essential surfaces in its complement,  cannot have an infinite number of Seifert surfaces.   
From work of Oertel in \cite{Oertel-02},  on a theorem of Jaco and Sedgwick,  it is finite  the number of compact essential surfaces of uniformly bounded genus (closed or with  boundary) in knot exteriors without essential genus one surfaces (closed or with boundary).  In \cite{Eudave-13},  in work related to the Kervaire conjecture,  Eudave-Mu\~noz proved that any odd number can be realized as the number of boundary components of an essential orientable connected surface properly embedded in a knot exterior. However, not necessarily in the same knot exterior and the genus is not arbitrarily large for the same number of boundary components.\\ 
Besides closed or longitudinal boundary slope,  a similar result exists for meridional surfaces in knot exteriors.  In fact,  in \cite{N-15} the author proved a prime knot exterior can have a collection of meridional essential surfaces with two boundary components and any positive genus.  However,  a collection of essential  surfaces in a knot exterior can be of arbitrarily large Euler characteristics not only because of large genus but also from the number of boundary components.  This was shown first by the author in \cite{N-16} with a collection of essential planar surfaces with arbitrarily large number of boundary components,  and,  even further,  in \cite{N-18} and in \cite{N-21}  with a collection of essential surfaces with independently large genus and number of boundary components in a satellite knot exterior and hyperbolic knot exterior,  respectively. \\

\noindent The paper is organized as follows: In section \ref{section:square double} we define a class of satellite knots that we will use throughout the paper.  In sections \ref{section: branched surface} and \ref{section: surfaces} we construct branched surfaces and use branched surface theory to prove that the surfaces they carry as in the statement of Theorem \ref{main} are essential in the corresponding knot exterior. In section \ref{section:non-parallel}, we extend Theorem \ref{main} by showing that there are knot exteriors which have a collection of many disjoint non-parallel longitudinal surfaces of each topological class of compact surface. Throughout this paper all manifolds are orientable,  all submanifolds are assumed to be in general position and we work in the smooth category.

\section{Square double of a knot}\label{section:square double}

In this section, we will construct a branched surface in the square double exteriors of composite knots, which we use in the proof of Theorem \ref{main}.

First we consider the 2-string tangle $T_1=(B_1; s_1\cup s_2)$ as in Figure \ref{figure:square tangle}(a) which we refer to as the square tangle,  following nomenclature in the literature \cite{Stoimenow-02}.

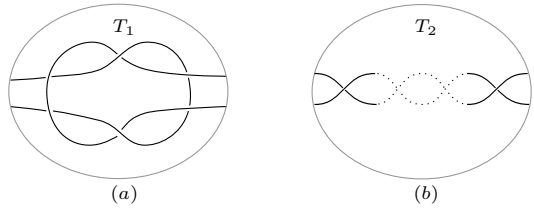
\begin{figure}[htbp]

\labellist
\small \hair 2pt

\scriptsize
\pinlabel $T_1$ at 80 105
\pinlabel $T_2$ at 288 105

\pinlabel $(a)$ at 80 -10 
\pinlabel $(b)$ at 288 -10 

\endlabellist

\centering
\includegraphics[width=0.55\textwidth]{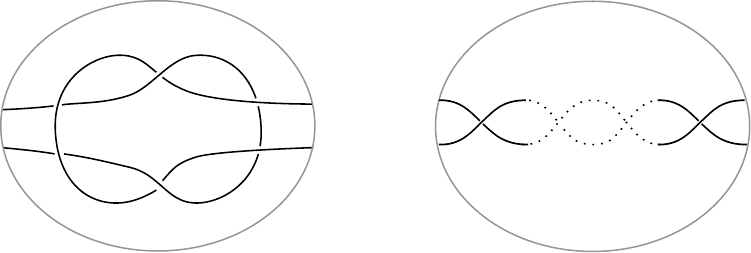}

\caption{: In (a) we have a diagram of the square tangle, denoted $T_1$; in (b) we have represented a  tangle of two twisted parallel arcs following the pattern of a knot $J$, denoted $T_2$.}
\label{figure:square tangle}
\end{figure}

Let $T'_2$ denote the $n/1$ rational tangle, where $n$ is an integer,  with $T'_2=(B_2; p_1\cup p_2)$. 
The arcs $p_1$ and $p_2$ co-bound a disk $D$ in $B_2$ with two arcs in $\partial B_2$,  say $a_1$ and $a_2$. Consider the operation on $T'_2$ where we assume the core of a regular neighborhood of $D$ in $B_2$ follows the pattern of a knot $J$. Hence, the image of the arcs $p_i,\, i=1, 2$, in $B_2$, follow this pattern. That is, through the closure of $p_i$ with an arc in $\partial B_2$ we obtain the knot $J$. We denote the resulting tangle as $T_2(n; J)$, or only as $T_2$,  and we assume that the arcs $a_i$, of $\partial D\cap \partial B_2$,  are in the boundary circle of the diagram disk of $T_2$.   (See Figure \ref{figure:square tangle}(b)). 

\begin{defn} 
For each integer $n$,  consider now the knot $K(n;J)$ obtained as the closure of $T_1$ with $T_2(n;J)$,  such that the boundary circle of their diagram disks,  as in Figure \ref{figure:square tangle},  are identified with opposite orientation and,  for each $i=1,2$,  the end points of $a_i$ are connected to $s_i$.  The knot $K(n;J)$ is referred to as a \textit{square double of $J$}.
\end{defn}

Note that the square knot is a square double of the unknot,  more exactly $K(0; \text{unknot})$.  In case the knot $J$ is non-trivial a square double of $J$ is a satellite knot with companion $J$.  In Figure \ref{figure:examples},  we have diagrams of two examples of square doubles, the square knot and a square double of a trefoil.  In Figure \ref{figure:square knot}, we have a schematic diagram illustration of a square double of $J$.

\begin{figure}[htbp]
	
	\labellist
	\small \hair 2pt
	
	\scriptsize
	
	\pinlabel (a) at 115 -10	 
	\pinlabel (b) at 445 -10 
	
	\endlabellist
	
	\centering
	\includegraphics[width=0.8\textwidth]{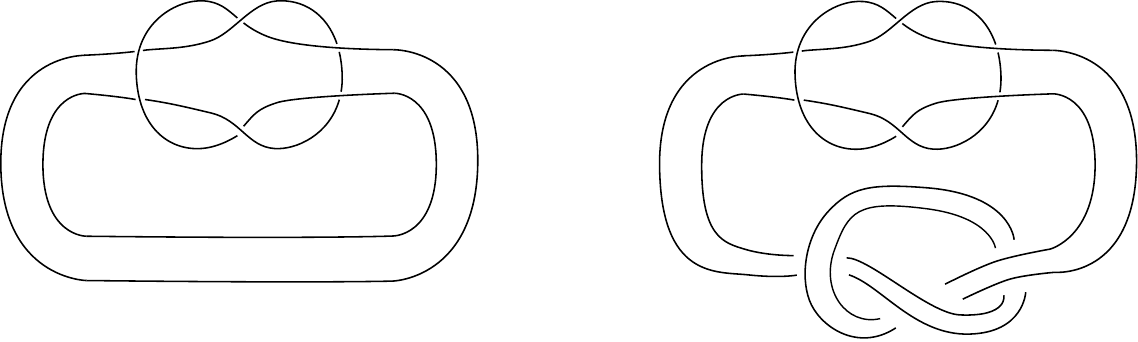}
	
	\caption{: (a) A diagram of the square knot, which is a square double of the unknot; (b) a diagram of a square double of the trefoil knot.}
	\label{figure:examples}
\end{figure}

\begin{figure}[htbp]
	
	\labellist
	\small \hair 2pt
	
	\scriptsize
	
	\pinlabel $T_1$ at 50 80 
	\pinlabel $T_2$ at 125 10 
	
	\endlabellist
	
	\centering
	\includegraphics[width=0.4\textwidth]{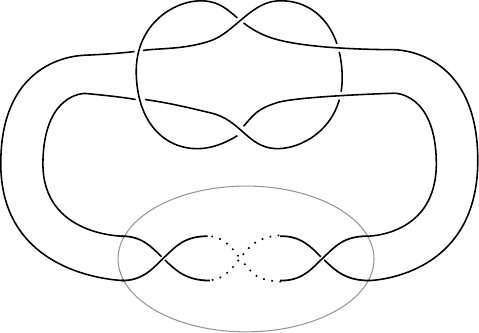}
	
	\caption{: Square double of a knot $J$.}
	\label{figure:square knot}
\end{figure}

\begin{lem}\label{lemma:ribbon}
A square double of a knot is a ribbon 2 knot. 
\end{lem}
\begin{proof}
Each arc $s_i$ cobounds with $a_i$ a disk $D_i$ in $B_1$, $i=1, 2$. In general position, $D_1$ and $D_2$ intersect each other in two arcs, $\alpha_1$ and $\alpha_2$, as illustrated schematically in Figure \ref{figure: disks in tangle}.

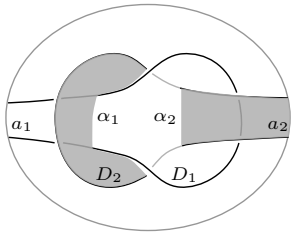
\begin{figure}[htbp]
	
	\labellist
	\small \hair 2pt
	
	\scriptsize
	\pinlabel $\alpha_1$ at 55 60 
	\pinlabel $a_1$ at 9 56
	\pinlabel $D_1$ at 95 30
	
	\pinlabel $\alpha_2$ at 85 60 
	\pinlabel $a_2$ at 145 56
	\pinlabel $D_2$ at 55 30
	
	\endlabellist
	
	\centering
	\includegraphics[width=0.3\textwidth]{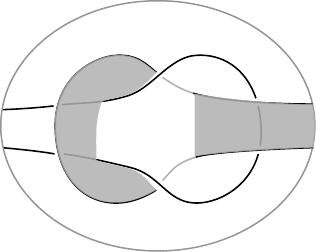}
	
	\caption{: Square tangle and transverse disks co-bounded by the arcs.}
	\label{figure: disks in tangle}
\end{figure}

Let $E$ be the immersed disk obtained by connecting $D_1$, $D$ and $D_2$ along $a_1$ and $a_2$. The immersed disk $E$ has boundary $K$ and self-intersects in two disjoint arcs which are ribbon singularities in $E$, that is the preimage of each of these arcs is two disjoint arcs in the preimage of $E$. Hence, as there are no other singularities, $E$ is a ribbon disk of $K$ with two ribbon singularities. 
\end{proof}

\begin{lem}\label{lemma:genus}
	The square double knot $K=K(2m+1; J)$, with $m$ an integer, has genus one.
\end{lem}
\begin{proof}
	Let $O_1$ (resp., $O_2$) be the annulus obtained by removing an open regular neighborhood of $\alpha_2$ (resp., $\alpha_1$) from $D_1$ (resp., $D_2$), as illustrated in Figure \ref{figure: disks in tangle}. Hence, by connecting $O_1$, $D$ and $O_2$ along $a_1$ and $a_2$ we obtain an embedded twice punctured disk $P$. From the ribbon disk $E$, as in Lemma \ref{lemma:ribbon}, the arcs $\alpha_1$ and $\alpha_2$ also separate an embedded disk containing $D$. Let $A$ be a swallow-follow annulus along this disk from $\partial O_1 - \partial D_1$ to $\partial O_2 - \partial D_2$ (denoted $\partial A_1$ and $\partial A_2$, respectively). As $2m+1$ is odd, the boundary components of $A$ are in the same side of the disk $E$. Hence, connecting $O_1$, $D$, $O_2$ and $A$ we obtain an embedded torus $F$ with boundary $K$. As $K$ is knotted, there is no compressing disk for $F$, otherwise $K$ would bound a disk. Hence, $F$ is essential and a Seifert surface of $K$.
\end{proof}

\begin{lem}\label{lemma:prime}
A square double of a non-trivial knot is prime. The only square double of the unknot that is composite is the square knot.
\end{lem}
\begin{proof}
The first part of this theorem is a consequence of a result of Lickorish \cite{Lickorish-81} stating that a knot with a 2-string prime decomposition is a prime knot. We recall that a tangle is said prime if it is essential and there are no local knots, that is no ball intersects the strings of the tangle in a single non-trivial arc. As it is observed also in \cite{Lickorish-81} the tangle $T_1$ is a 2-string prime tangle. The tangle $T_2$ is defined by two parallel arcs with a pattern of a non-trivial knot. Hence,  there are no local knots in $T_2$,  otherwise the strings could not be parallel, and $T_2$ is an essential tangle, otherwise the pattern $J$ would be of the unknot.  Therefore, the square double of a non-trivial knot has a 2-string prime tangle decomposition. Hence, from \cite{Lickorish-81} it is prime. \\ One other way to prove this statement is by using the wrapping number. The wrapping number of a square double of a non-trivial knot is 2,  as it can be seen by taking the meridian of the companion solid torus that is a regular neighborhood of $a_i$.  If there was a decomposing sphere, by taking its intersection with the companion torus and an innermost curve argument on the intersecting curves on the sphere, we would obtain a compressing disk for the companion torus or a meridian disk for the companion solid torus intersecting $K$ once,  contradicting  the existence of a meridian disk intersecting $K$ at two points.\\
For the second part of the theorem note first that if the pattern of the parallel arcs in $T_2$ is the unknot then the tangle is an integral rational tangle. We know that if $T_2$ is the 0 rational tangle we obtain the square knot, which is composite. From Eudave-Muñoz work \cite{EN-88}, if $r/s$ and $p/q$ are two rational tangles whose closure of $T_1$ return a composite knot then $|ps-qr|\leq 1$. Then, the only other possibilities for $T_2$ when closing $T_1$ to return a composite knot is if $T_2$ is a $1$ or $-1$ rational tangle.  However, these knots have genus 1, as stated in Lemma \ref{lemma:genus} for the knots $K(2m+1, J)$, where in this case $2m+1=\pm 1$ and $J$ is the unknot. As genus is additive under connected sum it cannot be 1 if the knot is composite. Therefore, these knots are also prime. Hence, the only square double of the unknot that is not prime is the square knot. 
\end{proof}

\section{Longitudinal branched surface in a square double exterior}\label{section: branched surface}

For the proof of Theorem \ref{main}, we use branched surface theory based on work of Oertel in \cite{Oertel-84} and of Floyd and Oertel in \cite{Floyd-Oertel}.  First, we revise the definition of a branched surface and a surface carried by a branched surface in the following paragraphs.\\

A \textit{branched surface} $R$ with generic branched locus is a compact space locally modeled as shown in Figure \ref{branchedmodel}(a). Hence, a union of finitely many compact smooth surfaces in a $3$-manifold $M$, glued together to form a compact subspace of $M$ respecting the local model, is a branched surface. We denote by $N=N(R)$ a fibered regular neighborhood of $R$ (embedded) in $M$, locally modeled as shown in Figure \ref{branchedmodel}(b). The boundary of $N$ is the union of three compact surfaces $\partial_h N$, $\partial_v N$ and $\partial M\cap \partial N$, where a fiber of $N$ meets $\partial_h N$ transversely at its endpoints and either is disjoint from $\partial_v N$ or meets $\partial_v N$ in a closed interval in its interior. We say that a surface $S$ is \textit{carried} by $R$ if it can be isotoped into $N$ so that it is transverse to the fibers. Furthermore, $S$ is carried by $R$ with \textit{positive weights} if $S$ intersects every fiber of $N$. If we associate a weight $w_i\geq 0$ to each component in the complement of the branch locus in $R$ we say that we have an $\textit{invariant measure}$ provided that the weights satisfy the \textit{branch equations} as in Figure \ref{branchedmodel}(c). Given an invariant measure on $R$, we can define a surface carried by $R$ with respect to the number of intersections between the fibers and the surface. We also note that if all weights are positive then the surface carried can be isotoped to be transverse to all fibers of $N$, and hence is carried with positive weights by $R$.

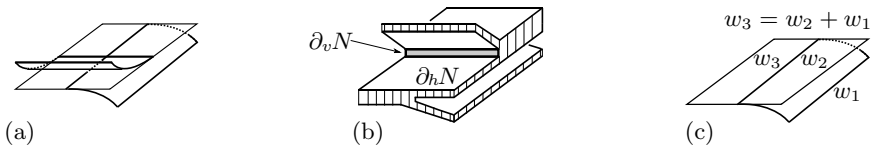
\begin{figure}[htbp]
	
	\labellist
	\small \hair 0pt
	\pinlabel (a) at 3 -5

	\pinlabel (b) at 173 -5

	\pinlabel (c) at 335 -5
	
	\pinlabel  $\partial \text{ }N$ at 207 23
	\pinlabel \tiny $h$ at 205 21

	\pinlabel  $\partial \text{ }N$ at 155 40
	\pinlabel \tiny $v$ at 153 37

	\pinlabel $w_3=w_2+w_1$ at 383 50
	\pinlabel $w_1$ at 407 14
	\pinlabel $w_2$ at 391 30
	\pinlabel $w_3$ at 368 30

	\endlabellist
	\centering
	\includegraphics[width=0.9\textwidth]{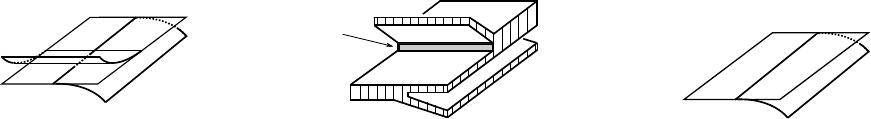}
	\caption{: Local model for a branched surface, in (a),  its regular neighborhood, in (b), and weights with branch equations, in (c).}
	\label{branchedmodel}
\end{figure}

In this section, we will construct a branched surface in square double exteriors of composite knots, which we use on the proof of Theorem \ref{main}.\\

Let $K=K(2m+1; J)$, with $m$ an integer,  be a square double of $J$,  a composite knot with summands $J_1$ and $J_2$, where we assume $J_1$ is as described in \cite{N-15}.  That is, $J_1$ has in its exterior meridional essential surfaces of any genus with two boundary components, which we denote $X_g$ for genus $g-1$.  Denote by $Q_i$ the ball intersecting $J$ in a single arc with pattern $J_i$. We assume,  after an ambient isotopy if needed, that $Q_i$ is in $B_2$,  and the intersection of $Q_i$ with $p_1\cup p_2$ is two parallel arcs with pattern $J_i$.  Let $S_i$ be the boundary sphere of $Q_i$. Let $A$ be the annulus as defined in the proof of Lemma \ref{lemma:genus}. The sphere $S_i$ intersects $A$ at two circles. These two circles cobound an annulus $U_i$ in $A$, and an annulus $V_i$ in $S_i$.  With $U_1$ as the boundary of a tubular neighborhood of an arc with pattern $J_1$, we can assume, after an isotopy if needed, that the boundary components of $X_g$ are also the two boundary circles of $U_1$.\\

 Consider the boundary components of $A$, $\partial_i A$, $i=1, 2$, as denoted in the proof of Lemma \ref{lemma:genus}. Without loss of generality, we assume that a path in $A$ from $\partial_1 A$ must first intersect $S_1$ before $S_2$ and $S_2$ before $\partial_2 A$. Let $A_1$ (resp., $A_2$) be the annulus in $A$ between $\partial_1 A$ and $U_1$ (resp., $\partial_2 A$ and $U_2$). Finally, we denote by $A'$ the annulus in $A$ from $U_1$ and $U_2$. (See Figure \ref{figure:annulus A}.)

\begin{figure}[htbp]
	
	\labellist
	\small \hair 2pt
	
	\scriptsize
	
	\pinlabel $\partial_1A$ at 10 30
	\pinlabel $A_1$ at 50 60 
	\pinlabel $U_1$ at 120 60
	\pinlabel $A'$ at 190 60
	\pinlabel $U_2$ at 270 60
	\pinlabel $A_2$ at 340 60
	\pinlabel $\partial_2A$ at 380 30
	
	\pinlabel $S_1$ at 120 100
	\pinlabel $S_2$ at 270 100

	\endlabellist
	
	\centering
	\includegraphics[width=0.45\textwidth]{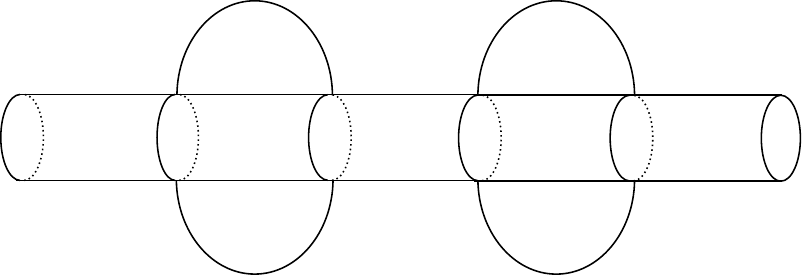}
	
	\caption{: Partition of the annulus $A$ with respect to the intersection with $S_1$ and $S_2$.}
	\label{figure:annulus A}
\end{figure}
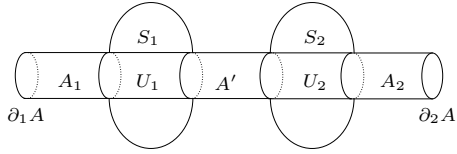

Let $P$ be the twice punctured disk as defined in the proof of Lemma \ref{lemma:genus}. 
 Consider the union of $P$, $A_1$, $U_1$,  $X_g$,  $V_1$,  $A'$, $U_2$, $V_2$ and $A_2$ denoted by $R_g$. We smooth the space $R_g$ on the intersection of these union parts as follows: At $\partial_i A$, $i=1, 2$, there is no singularity in $R_g$; hence it is smoothed with $P$. The annulus $U_i$ at $A_i\cap U_i$, $i=1,2$ respectively,  is smoothed towards $A_i$. The annulus $V_i$ at $V_i\cap A_i$ is smoothed towards $A_i$, and at $V_i\cap U_i$ is smoothed towards $U_i$, $i=1,2$. The surface $X_g$ at $X_g\cap A_1$ is smoothed towards $A_1$,  and at $X_g\cap A'$ is smoothed towards $A'$.  The annulus $A'$ at  $A'\cap U_1$  is smoothed towards $U_1$ (and $X_g$) and at $A'\cap U_2$ is smoothed towards $V_2$ (similarly to $U_2$). The Figure \ref{figure:branched surface} provides an illustration of $R_g$.

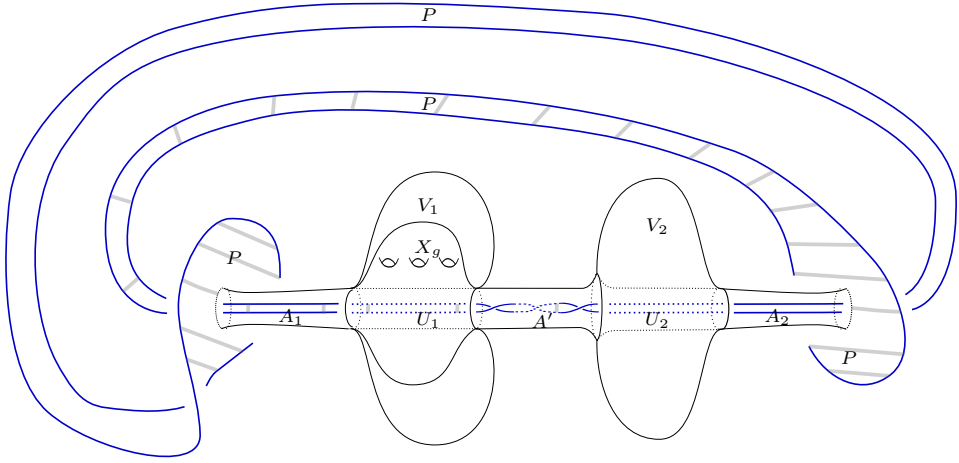
\begin{figure}[htbp]
	
	\labellist
	\small \hair 2pt
	
	\scriptsize
	\pinlabel $P$ at 80 70
	\pinlabel $A_1$ at 100 49
	
	\pinlabel $V_1$ at 148 88
	\pinlabel $X_g$ at 148 73
	\pinlabel $U_1$ at 148 48
	
	\pinlabel $A'$ at 188 48
	
	\pinlabel $V_2$ at 228 81
	\pinlabel $U_2$ at 228 48
	
	\pinlabel $A_2$ at 270 49
	\pinlabel $P$ at 295 35
	
	\pinlabel $P$ at 148 155
	\pinlabel $P$ at 148 124

	\endlabellist
	
	\centering
	\includegraphics[width=1\textwidth]{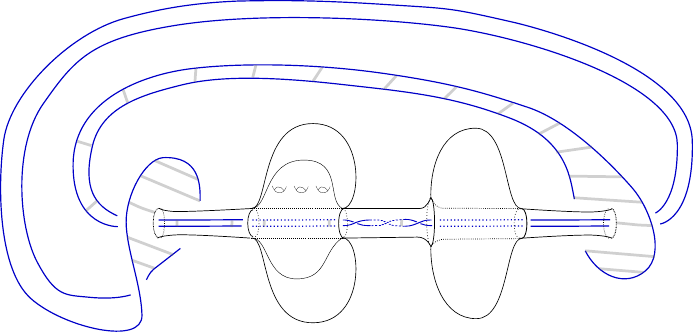}
	
	\caption{: An illustration of the branched surface $R_g$.}
	\label{figure:branched surface}
\end{figure}

Note that some sections of $R_g$ branch into three sections,  as $A_1$ with $V_1,  X_g,   U_1$,  or two branches into two branches,  as $X_g,  U_1$ and $V_1,  A'$,  as illustrated in Figure \ref{figure:branched surface},  which is not in conformity with the local model of branched surface.  However,  a small isotopy corrects this for the purpose of the local model. For convenience on the number of sections involved, we continue with the current choice of sections up to a small isotopy correction. So, from this construction and up to a small isotopy, the space $R_g$ is a branched surface with sections denoted by $P$, $A_1$, $U_1$, $X_g$,  $V_1$,  $A'$, $U_2$, $V_2$ and $A_2$.

\section{Longitudinal essential surfaces from a branched surface}\label{section: surfaces}

In this section we prove Theorem \ref{main}.  We will do so by proving that the surfaces carried by $R_g$ are essential. \\

We recall that the components on the complement of the branched locus of $R_g$ are $P$, $A_1$,  $U_1$, $X_g$, $V_1$,  $A'$, $U_2$, $V_2$ and $A_2$.  We denote the weights of each component on an invariant measure for $R_g$ as $W_P$,  $W_{A_1}$, $W_{U_1}$,  $W_{X_g}$,  $W_{V_1}$,  $W_{A'}$,  $W_{U_2}$,  $W_{V_2}$ and $W_{A_2}$, respectively.\\

Now we define surfaces $F_g^n$ carried by $R_g$ with genus $g$ and any number $n$ of longitudinal boundary components.  We define $F_g^1$ as the surface carried by $R_g$ with invariant measure $W_P=1$, $W_{A_1}=1$,  $W_{U_1}=0$, $W_{X_g}=1$,  $W_{V_1}=0$, $W_{A'}=1$, $W_{U_2}=1$, $W_{V_2}=0$ and $W_{A_2}=1$.  Let us define $F_g^2$ as the surface carried by $R_g$ with invariant measure $W_P=2$, $W_{A_1}=2$, $W_{U_1}=0$, $W_{X_g}=1$,  $W_{V_1}=1$, $W_{A'}=0$, $W_{U_2}=1$, $W_{V_2}=1$ and $W_{A_2}=2$.  For $n\geq 3$, we define $F_g^n$ as the surface carried by $R_g$ with invariant measure $W_P=n$, $W_{A_1}=n$, $W_{U_1}=n-2$, $W_{X_g}=1$,  $W_{V_1}=1$, $W_{A'}=n-2$, $W_{U_2}=1$, $W_{V_2}=n-1$ and $W_{A_2}=n$.  In Figure \ref{figure:schematic} we have a schematic illustration of the surfaces $F_g^n$ for $n\geq 3$.\\

\begin{lem}\label{lemma:orientable}
The surface $F_g^n$ is orientable and connected.
\end{lem}
\begin{proof} For convenience of notation, without loss of generality, we consider one side of $P$ to be the positive ``$+$'' side and the other side to be the negative ``$-$'' side.
	If $n=1$, we connect the non-peripheral boundary components of $P$ respecting the orientation of $P$.  If $n=2$ we connect the non-peripheral boundary components of $P_1$ to the ones of $P_2$ always from the $+$ or $-$ side of $P_1$ to the corresponding $+$ or $-$ side of $P_2$. Hence, we obtain an orientable connected surface. \\  
	Suppose now that $n\geq 3$. As it can be observed from the schematic representation of $F_g^n$ in Figure \ref{figure:schematic},  all the odd indexed $P_i$'s are connected in linear order, and similarly the even indexed are connected in linear order.  Here, we respect the positive and negative sides of the $P_i$'s. The positive side of $P_1$ is connected to the negative side of $P_2$, and the positive side of $P_{n-1}$ is connected to the negative side of $P_n$.  So, we connect the odd indexed to the even indexed sequences of $P_i$'s to opposite sides at the beginning and at the end of the odd and even sequences, forming a loop. Hence,  $F_g^n$ is connected and orientable. 
\end{proof}

\begin{figure}[htbp]
	
	\labellist
	\small \hair 2pt
	
	\scriptsize
	\pinlabel $P_1$ at 33 99
	\pinlabel $P_1$ at 280 112

	\pinlabel $P_2$ at 33 84
	\pinlabel $P_2$ at 330 110	
	\pinlabel $P_3$ at 33 69
	\pinlabel $P_3$ at 253 90
	\pinlabel $P_4$ at 357 90

	\pinlabel $P_n$ at 35 3
	\pinlabel $P_n$ at 277 -2

	\pinlabel $P_{n-1}$ at 41 18	
	\pinlabel $P_{n-1}$ at 338 -2	
	
	\pinlabel $P_{n-2}$ at 41 33
	\pinlabel $P_{n-2}$ at 252 12

	\pinlabel $P_{n-3}$ at 358 12

	\pinlabel $+$ at 60 106	
	\pinlabel $-$ at 60 100
	\pinlabel $+$ at 152 106	
	\pinlabel $-$ at 152 100

	\pinlabel $+$ at 152 91	
	\pinlabel $-$ at 152 85	
	\pinlabel $+$ at 60 91	
	\pinlabel $-$ at 60 85

	\pinlabel $+$ at 60 77	
	\pinlabel $-$ at 60 70	
	\pinlabel $+$ at 152 77	
	\pinlabel $-$ at 152 70

	\pinlabel $+$ at 60 24	
	\pinlabel $-$ at 60 17		
	\pinlabel $+$ at 152 24	
	\pinlabel $-$ at 152 17		

	\pinlabel $+$ at 60 39	
	\pinlabel $-$ at 60 32	
	\pinlabel $+$ at 152 39	
	\pinlabel $-$ at 152 32

	\pinlabel $+$ at 60 9	
	\pinlabel $-$ at 60 2
	\pinlabel $+$ at 152 9	
	\pinlabel $-$ at 152 2

	\endlabellist
	
	\centering
	\includegraphics[width=0.85\textwidth]{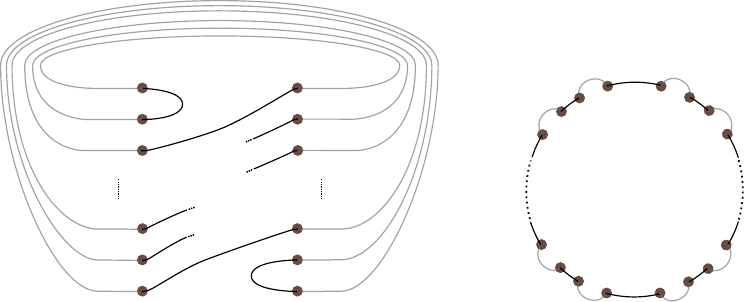}
	
	\caption{: On the left, there is a schematic representation of $F_g^n$,  $n\geq 3$,  on how the components corresponding to sections of $R_g$ connect.  On the right, there is a schematic representation on how the copies of $P$ are connected in $F_g^n$,  for $n\geq 3$ odd.}
	\label{figure:schematic}
\end{figure}
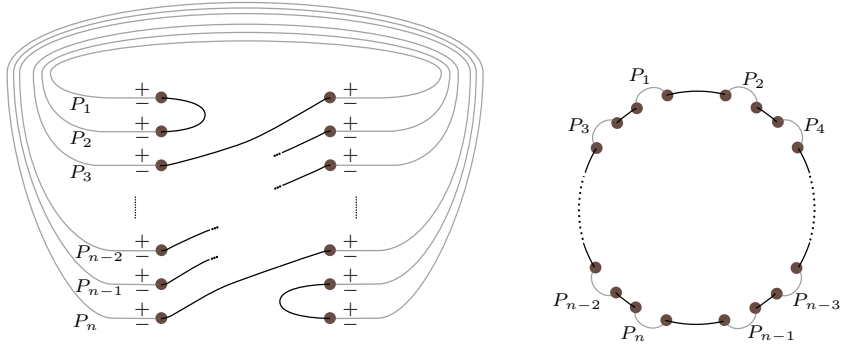

\begin{lem}\label{lemma:slope}
	The surface $F_g^n$ has longitudinal slope.
\end{lem}
\begin{proof}
Each boundary component of $F_g^n$ is the boundary component of a copy of $P$ in $\partial E(K)$. We recall that the section $P$ of $R_g$ is a region of the Seifert surface $F$ as in the proof of Lemma \ref{lemma:genus}. As a Seifert Surface has longitudinal slope, the boundary slope of $F_g^n$ is longitudinal.
\end{proof}

\begin{lem}\label{lemma:genus and boundary}
	The surface $F_g^n$ has genus $g$ and $n$ boundary components.
\end{lem}
\begin{proof}
	From the construction of $F_g^n$,   its boundary components come from the longitudinal boundary component of each copy of $P$.  As for $F_g^n$ we use $n$ copies of $P$, and all the other components have their boundaries connected to each other, we have that $F_g^n$ has $n$ boundary components.  In the definition of $F_g^n$,  the components other than copies of $P$ and $X_g$ are always annuli.  The Euler characteristic, $\chi$, of an annulus is $0$ and $\chi$ is additive under gluing components along their boundaries. Hence,  the Euler characteristic of $F_g^n$ is $n\times \chi(P)+\chi(X_g)$, which is $-n+(2-2(g-1)-2)= 2-2g-n$.  As $F_g^n$ is orientable and connected (Lemma \ref{lemma:orientable}) with $n$ boundary components,  it comes from the Euler characteristics formula of a $n$-punctured orientable connected surface that it has genus $g$,  completing the proof of the lemma. 
\end{proof}

The following concepts are relevant for the definition of incompressible branched surface. Let $R$ denote a branched surface in a 3-manifold $M$, with regular neighbood $N$ in $M$. A \textit{disc of contact} is a disc $O$ embedded in $N$ transverse to fibers and with $\partial O\subset \partial_v N$. A \textit{half-disc of contact} is a disc $O$ embedded in $N$ transverse to fibers with $\partial O$ being the union of an arc in $\partial M\cap \partial N$ and an arc in $\partial_v N$. A \textit{monogon} in the closure of $M-N$ is a disc $O$ with $O\cap N=\partial O$ which intersects $\partial_v N$ in a single fiber. (See Figure \ref{monogon}.) We say a branched surface $R$ in M contains a \textit{Reeb component} if $R$ carries a compressible torus or properly embedded annulus, transverse to the fibers of $N$, bounding a solid torus in $M$. (This is a weaker version of the definition of Reeb component in \cite{Oertel-84} by Oertel.)

\begin{figure}[htbp]
	
	\labellist
	\small \hair 0pt
	\pinlabel (a) at 3 -6
	
	\scriptsize
	\pinlabel monogon at 100 30
	
	\small
	\pinlabel (b) at 173 -6
	
	\scriptsize
	\pinlabel  \text{disk of} at 240 40
	
	\pinlabel \text{contact} at 240 34
	
	\endlabellist
	\centering
	\includegraphics[width=0.6\textwidth]{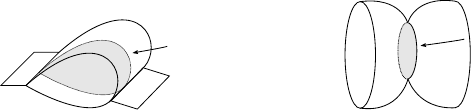}
	\caption{: Illustration of a monogon and a disk of contact on a branched surface.}
	\label{monogon}
\end{figure}
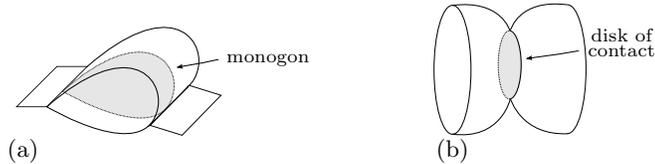

We recall that a branched surface $R$ embedded in a 3-manifold $M$ is said to be incompressible if it satisfies the following three properties:

\begin{itemize}
	\item[(i)] $R$ has no disk of contact or half-disks of contact;
	\item[(ii)] $\partial_h N$ is incompressible and  boundary incompressible in the closure of $M-N$, where a boundary compressing disk is assumed to have boundary defined by an arc in $\partial M$ and an arc in $\partial_h N$;
	\item[(iii)] There are no monogons in the closure of $M-N$. 
\end{itemize}

and it is said to be \textit{without Reeb components} if it satisfies the following property: 

\begin{itemize}
\item[(iv)] $R$ doesn't carry a Reeb component.
\end{itemize}

The following theorem, proved by Oertel in \cite{Oertel-84}, let us infer if a surface carried by a branched surface is essential. Note that condition (iv), $R$ not carrying a torus or an annulus cutting a solid torus from $M$, implies the non-existence of Reeb components in the sense of Oertel \cite{Oertel-84}.

\begin{thm}[Oertel, \cite{Oertel-84}]\label{Oertel}
If $R$ is an incompressible branched surface without Reeb components (i.e. satisfies (i)-(iv)) and $R$ carries some surface with positive weights then any surface carried by $R$ is essential.
\end{thm}

\begin{lem}\label{lemma:incompressible}
The branched surface $R_g$,  $g\geq 1$,  is incompressible and without Reeb components.
\end{lem}
\begin{proof}
First we observe that $R_g$ doesn't carry a Reeb component.  In fact,  if $R_g$ would carry a torus $T$,  this torus couldn't be transverse to the regular neighborhood of sections of $R_g$ with boundary components, or of sections of $R_g$ of genus bigger than $1$.  In this case,  it could only be transverse to regular neighborhoods of sections $A_1$, $A'$,  $X_g$ (in case $g=1$),  $U_1$,  $V_1$, $U_2$, $V_2$ and $A_2$.  Only one boundary component of $A_1$ and of $A_2$ is connected to these other sections. Hence, $A_1$ and $A_2$ cannot contribute for $R_g$ to carry a torus.  The annulli $U_2$ and $V_2$ are smoothed towards $A_2$.  As $W_{A_2}=0$ in an invariant measure for a torus carried by $R_g$,  we have that necessarily $W_{U_2}=W_{V_2}=0$ as well.  Similarly,  $U_1$, $V_1$ and $X_g$ are smoothed towards $A_1$.  As $W_{A_1}=0$ in an invariant measure for a torus carried by $R_g$,  we also have that  $W_{U_1}=W_{V_1}=W_{X_g}=0$.  Then $A'$ is the only section left whose regular neighborhood could be transverse to a torus.  As $A'$ is smoothed towards $U_1$,  $X_g$ and $V_2$, similarly we have that $W_{A'}=0$ in an invariant measure of a torus carried by $R_g$.  Therefore,  $R_g$ does not carry a torus.\\

If $R_g$ would carry a properly embedded annulus in the exterior of $K$,  $E(K)$,  this annulus would have to be transverse to the regular neighborhood of sections of $R_g$ with boundary components,  that is $P$. However, $P$ has negative Euler characteristics, and the remaining sections of $R_g$ have non-positive Euler characteristics. Hence, any surface carried by $R_g$ with positive weight on $P$ has negative Euler characteristics, which cannot be an annulus.  Therefore, $R_g$ does not carry an annulus.\\

Now we prove that $R_g$ is incompressible in $E(K)$. First observe that there are no (half) disks of contact as no circle on the branched locus of $R$ bounds a disk in $\partial_h N(R_g)$ and there are no properly embedded arcs on the branched locus of $R_g$. The space $N(R_g)$ decomposes $E(K)$ into four components: a component cut from $E(K)$ by $X_g$ and $V_1$, denoted $M_1$; a component cut from $E(K)$ by $X_g$ and $U_1$,  denoted $M_1'$; a component cut from $E(K)$ by $U_2$, and $V_2$, denoted $M_2$; a component cut from $E(K)$ by all sections but $X_g$ denoted $M$.\\

There are six components in $\partial_v N(R_g)$: one in $M_1$,  two in $M_1'$, one in $M_2$ and two in $M$.  The component of $\partial_v N(R_g)$ in $M_1$ corresponds to a non-separating circle in $X_g\cup V_1$.  Hence, a monogon in $M_1$  is a compressing disk for $X_g\cup V_1$ in $M_1$,  but this contradicts $X_g$ being essential in the exterior of $J_1$.  A monogon in $M_1'$ is impossible as the boundary of a monogon disk in $M_1'$ would have to intersect one boundary component of the annulus $U_1$ but be disjoint from the other boundary component,  as both correspond to components of $\partial_v N(R_g)$.  Therefore,  there are no monogons in $M_1$ and in $M_1'$.  Similarly, as $\partial_v N(R_g)$ in $M_2$ corresponds to a non-separating circle in $U_2\cup V_2$ and $U_2\cup V_2$ is essential in the exterior of $J_2$,   there are no monogons in $M_2$.\\
Each curve of the branched locus corresponding to the two components of $\partial_v N(R_g)$ in $M$ are non-separating in $\partial M$. Hence,  a monogon in $M$ corresponds to a compressing disk for $\partial M$ in $M$. Let $c_1$ be the curve $V_1\cap A'$ and $c_2$ be the curve $V_2\cap A'$. Suppose there is a monogon $O_1$ corresponding to $c_1$. Then the surface $Y_1$ defined by $P$, $A_1$, $V_1$, $A'$, $V_2$ and $A_2$ has a compressing disk $O_1$.  But $Y_1$ is a genus one surface bounded by $K$, which is incompressible because $K$ is non-trivial and hence we have a contradiction. Suppose there is a monogon $O_2$ corresponding to $c_2$.  Then the surface defined by $P, \, A_1, \, U_1, \, A', \,  U_2$ and $A_2$ has compressing disk $O_2$.  But $Y_2$ is a genus one surface bounded by $K$,  so it is a minimal genus Seifert surface and hence it is essential, contradicting $O_2$ being a compressing disk for $O_2$.  Therefore, there are no monogons in $M$.\\

We proceed to prove that $\partial_h N(R_g)$ is (boundary) incompressible in the exterior of $K$. In $M_1$,  $\partial_h N(R_g)$ corresponds to the complement of $A_1 \cap X_g \cap V_1$ in $X_g\cup V_1$. Hence,  as $X_g$ is incompressible in $M_1$ we have that $\partial_h N(R_g)\cap M_1$ is incompressible.  In $M_1'$, we argue similarly that  $\partial_h N(R_g)\cap M_1'$ is incompressible.  In $M_2$, $\partial_h N(R_g)\cap M_2$ corresponds to the complement of $A_2 \cap U_2 \cap V_2$ in $U_2\cup V_2$, which corresponds to the exterior in $S^3$ of the knot $J_2$. Hence, $\partial_h N(R_g)\cap M_2$ is incompressible in $M_2$. At last, in $M$, $\partial_h N(R_g)\cap M$ corresponds to the complement of $\partial A'$ in $\partial M$, which is represented by $A'$ and $P$, after an ambient isotopy. An essential simple closed curve in $A'$ corresponds to a meridian of the satellite torus of $K$ (with pattern $J_1\#J_2$), hence there is no compressing disk for $A'$ in $M$. The twice punctured disk $P$ is essential in $M$, otherwise a (boundary) compressing disk would also be a (boundary) compressing disk for the Seifert surface defined by $P\cup A$  (see also the end of proof for Lemma \ref{lemma:prime}).\\

Hence, $\partial_h N(R_g)$ is incompressible and boundary incompressible in $E(K)$, and there are no Reeb components,  monogons or disks of contact  in $R_g$.  Therefore, $R_g$,  $g\geq 1$, is an incompressible branched surface without Reeb components.
\end{proof}

%We observe that in Lemma \ref{lemma: boundaries} we can lift the restriction of $F$ having exactly two boundary components in $N(s)$ to any (even) number. 

\begin{proof}[Proof of Theorem \ref{main}.]
Let $K=K(2m+1; J)$, with $m$ an integer,  be a square double of a composite knot $J$ and $R_g$, $g\geq 1$,  as defined in section \ref{section: branched surface}.  The collection of such knots is infinite.  In fact,  the choice of companion $J$ can be for arbitrarily large bridge number,  and by work of Schubert,  so is the bridge number of a corresponding satellite $K$.\\
Now we prove that the knots $K$ are as in the statement of the theorem.  From Proposition \ref{lemma:incompressible} we have that $R_g$, $g\geq 1$, is an incompressible branched surface without Reeb components.  We also have that the surfaces $F_g^n$, $n\geq 3$, are carried with positive weights by $R_g$. Hence, we are under the conditions of Theorem \ref{Oertel}.  It then follows that all surfaces carried by $R_g$ are incompressible and boundary incompressible in $E(K)$. Therefore, as $F_g^n$, $n\geq 1$,  is carried by $R_g$, $g\geq 1$, and is orientable (Lemma \ref{lemma:orientable}),  it is essential in $E(K)$. From Lemma \ref{lemma:slope}, the surface $F_g^n$ has longitudinal boundary slope. From Lemma \ref{lemma:genus and boundary}, the surface $F_g^n$, $g\geq 1,\, n\geq 1$,  has genus $g$ and $n$  boundary components. Then,  as we wanted to prove,  each compact orientable connected surface of positive genus with boundary has an essential proper embedding into the exterior of $K$.
\end{proof}

If we apply and follow the proof of Theorem 2 of \cite{N-18} with the collection of knot exteriors and essential surfaces of Theorem \ref{main} of this paper,  we can conclude that there are hyperbolic 3-manifolds with torus boundary each with a collection of longitudinal essential surfaces of independently unbounded genus and any number of boundary components.\\ 
As a final remark,  note that the knots $J_1$ or $J_2$ can be chosen with closed essential surfaces in their exteriors of every genus or meridional essential surfaces of every genus and (even) number of boundary components.  (See \cite{N-18} for instance.) So,  in this case,  the exterior of $K$,  as in the proof of Theorem \ref{main},  has (meridional) planar essential surfaces in its exterior but with the number of boundary components divisible by 4.  Furthermore,  in this case,  the exterior of $K$ has in fact all compact surfaces of positive genus essentially embedded, closed and with boundary.  

%\section{Acknowledgement} The author would like to thank Cameron Gordon and John Luecke for discussions on this topic.

\section{Arbitrarily many disjoint surfaces of any genus and number of boundary components.}\label{section:non-parallel}

There has been recent interest on the number of disjoint non-parallel punctured tori in a hyperbolic knot exterior in $S^3$. For once puntured tori in a knot exterior it has been proved by Valdez-Sánchez \cite{Valdez-Sanchez-19} that this number is at most five and that this bound is optimal. For higher number of boundary components, in \cite{Eudave-Teragaito-25} Eudave-Muñoz and Teragaito gave examples of hyperbolic knot exteriors with five non-parallel essential twice punctured tori, while in \cite{Aranda-Losada-Viorato-23} Aranda,  Ramírez-Losada and Rodríguez-Viorato proved that the number of  non-parallel essential twice punctured tori in hyperbolic exteriors is at most six.\\
In this section we show that for satellite knots there is no general bound on the number of disjoint non-parallel compact surfaces for any positive genus and any number of boundary components,  as stated in Theorem \ref{theorem:disjoint}.

\begin{thm}\label{theorem:disjoint}
For each positive integer $s$ there are infinitely many knots in the 3-sphere,  where each of which has in its exterior $s$ disjoint non-parallel longitudinal essential surface of any positive genus and any number of boundary components. 
\end{thm}
\begin{proof}
Let $s$ be a positive integer greater than 1. The case $s=1$ is Theorem \ref{main}. For the proof of this theorem we follow the same procedure as in the proof of Theorem \ref{main}.  Let $K=K(2m+1; J)$ be a square double of a composite knot $J$,  as defined in section \ref{section:square double}, where we choose the summand $J_2$ such that is has a collection of $s$ nested non-parallel essential properly embedded annulli with meridional boundary in its exterior.  For instance, we can choose $J_2$ such that it is a connected sum of $s$ knots.\\
Let $U_2^i$, $i=0, \ldots, s$, be a collection of such annulli sharing boundary components with $U_2$, with $U_2^0=U_2$. Reconstructing a branched surface similarly to $R_g$ but replacing $U_2$ with $U_2^i$, for some  $i\in \{0, \ldots, s\}$, we obtain a branched surface $R_{g;i}$ where each section $U_2^i$ is smoothed as $U_2$ is smoothed. By replacing $U_2$ by $U_2^i$ in Figure \ref*{figure:branched surface} we have an illustration of $R_{g;i}$.\\
Following the proof of Theorem \ref{lemma:incompressible} we have that $R_{g;i}$ is an incompressible branched surface and without Reeb components. By replacing $U_2$ of $R_g$ with $U_2^i$ of $R_{g;i}$ in section \ref{section: surfaces}, we define the surfaces $F_{g;i}^n$. For $n\geq 3$, the surface $F_{g;i}^n$ is carried with positive weights by $R_{g;i}$. Hence, as for $R_g$, $R_{g;i}$ carries some surface with positive weights. Therefore, from Theorem \ref*{Oertel}, every surface carried by $R_{g;i}$ is essential, in particular every surface $F_{g;i}^n$ is essential.\\
We observe now that the surfaces $F_{g;i}^n$, for $i=0, \ldots , s$, after a small isotopy if necessary, are disjoint and  non-parallel. In fact, let us consider a pair $F_{g;i}^n$ and $F_{g;i+1}^n$. We push $F_{g;i+1}^n$ in the direction of the branched surface section $V_2$ towards $U_2^{i+1}$, extended to all surface,  such that $F_{g;i+1}^n$ becomes disjoint from $F_{g;i}^n$. This is possible because the surfaces are orientable, and hence have two sides. If $n$ is even the surfaces become nested, if $n$ is odd they are not nested but follow each other in a chosed direction. In Figure \ref{figure:disjoint} we have a schematic representation of these surfaces.

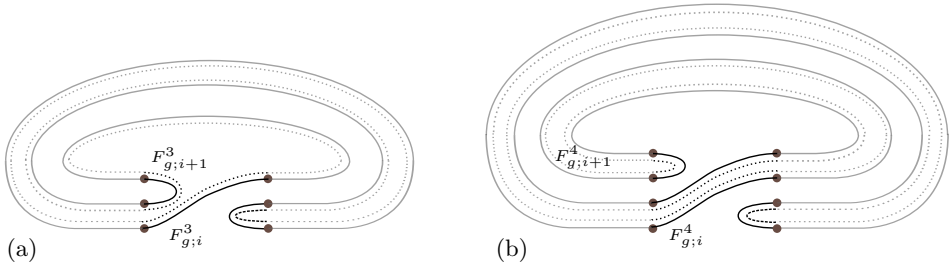
\begin{figure}[htbp]
	\labellist
	\small \hair 0pt
	\pinlabel (a) at 15 -8
	\tiny
	\pinlabel $F_{g;i+1}^3$ at 110 47
	\pinlabel $F_{g;i}^3$ at 114 0 
	\pinlabel $F_{g;i+1}^4$ at 353 48
	\pinlabel $F_{g;i}^4$ at 415 0 
	
	\small
	\pinlabel (b) at 310 -8
	\endlabellist
	\centering
	\includegraphics[width=\textwidth]{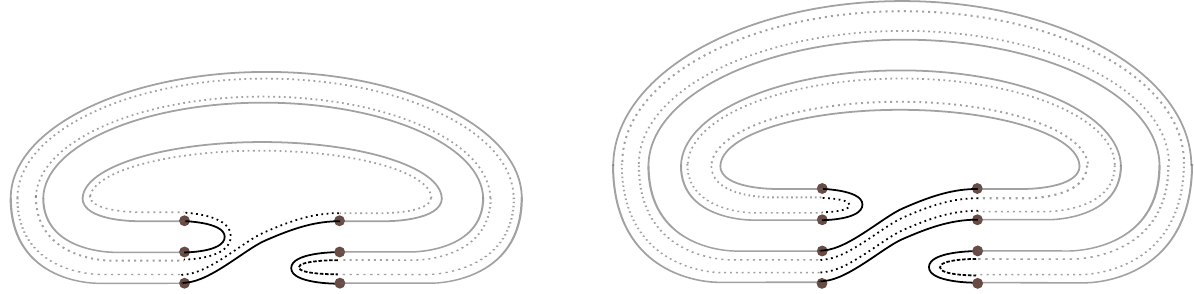}
	\caption{: A schematic representation of surfaces $F_{g;i}^n$ and $F_{g;i+1}^n$, in dashed line, for $n$ odd, in (a), and for $n$ even, in $(b)$, following the same notation as in Figure \ref{figure:schematic}. Note that for $n$ even the surfaces are nested, as they are separating, and for $n$ odd they are consecutive following a direction, as they are non-separating in the knot exterior.}
	\label{figure:disjoint}
\end{figure}

Following with the same procedure for each pair of surfaces, we obtain a collection of $s$ disjoint longitudinal essential surfaces of genus $g$ and $n$ boundary components in the exterior of $K$.     
\end{proof}

\section*{Acknowledgement}
The author would like to thank Cameron Gordon and John Luecke for conversations in this topic,  and Román Aranda for bringing to attention of the author and discussing the question addressed in section \ref{section:non-parallel}.

\end{document}